\documentclass[a4paper,10pt]{amsart}
\usepackage{amsmath,amssymb,amsthm,amscd}

\newcommand{\add}{\mathrm{add}}
\newcommand{\Add}{\mathrm{Add}}
\newcommand{\Prod}{\mathrm{Prod}}
\newcommand{\Gen}{\mathrm{Gen}}
\newcommand{\rmod}{\mathrm{Mod-}}
\newcommand{\rfmod}{\mathrm{mod-}}

\DeclareMathOperator{\End}{End}
\DeclareMathOperator{\Ker}{Ker}
\DeclareMathOperator{\Img}{Im}

\newcommand{\mapr}[1]{\xrightarrow{#1}}

\newcommand{\exs}[5]{0\rightarrow #1 \mapr{#2} #3 \mapr{#4} #5 \rightarrow 0}

\newcommand{\Ext}[3]{\mbox{Ext}^1_{#1}\,(#2,#3)}

\newcommand{\Exti}[4]{\mbox{Ext}^{#1}_{#2}\,(#3,#4)}

\newcommand{\Tor}[3]{\mbox{Tor}_1^{#1}\,(#2,#3)}

\DeclareMathOperator{\HomOp}{Hom}

\newcommand{\Hom}[3]{\HomOp_{#1}(#2,#3)}

\newcommand{\Fcal}{\ensuremath{\mathcal{F}}}

\newcommand{\Scal}{\ensuremath{\mathcal{S}}}

\newcommand{\p}{\ensuremath{\mathbf{p}}}
\newcommand{\q}{\ensuremath{\mathbf{q}}}
\newcommand{\tube}{\ensuremath{\mathbf{t}}}

\theoremstyle{plain}
\newtheorem{thm}{Theorem}
\newtheorem{prop}[thm]{Proposition}
\newtheorem{lem}[thm]{Lemma}
\newtheorem{cor}[thm]{Corollary}
\newtheorem{ex}[thm]{Example}
\theoremstyle{definition}

\theoremstyle{remark}
\newtheorem*{rem}{Remark}

\begin{document}
\title{Baer and Mittag-Leffler modules over tame hereditary algebras}
\author{\textsc{Lidia Angeleri H\" ugel}}
\address{Dipartimento di Informatica e
Comunicazione\\ Universit\`a degli Studi dell'Insubria\\
Via Mazzini 5, I - 21100 Varese, Italy}
\email{lidia.angeleri@uninsubria.it}
\author{\textsc{Dolors Herbera}}
\address{Departament de Matem\`atiques\\
Universitat Aut\`onoma de Barcelona\\E-08193 Be\-lla\-te\-rra
(Barcelona), Spain} \email{dolors@mat.uab.cat}
\author{\textsc{Jan Trlifaj}}
\address{Charles University, Faculty of Mathematics and Physics, Department of Algebra \\
Sokolovsk\'{a} 83, 186 75 Prague 8, Czech Republic}
\email{trlifaj@karlin.mff.cuni.cz}
\thanks{Part of this research was done during the visit of the authors to CRM Barcelona in September 2006
supported by the Research Programme on Discrete and Continuous Methods of Ring Theory. First author partially
supported by  Universit\`a di Padova, Progetto di Ateneo
CDPA048343, and by PRIN 2005 "Prospettive in teoria degli anelli, algebre di Hopf e categorie di moduli".
First two authors partially supported by the DGI and the
European Regional Development Fund, jointly, through Project
 MTM2005--00934, and by the Comissionat per Universitats i Recerca
of the Generalitat de Ca\-ta\-lunya, Project 2005SGR00206. Third author
 acknowledges support by GA\v CR 201/06/0510 and MSM 0021620839.}
\date{\today}

\begin{abstract} We develop a structure theory for two classes of infinite dimensional modules over tame hereditary algebras: the Baer modules, and the Mittag--Leffler ones. A right $R$--module $M$ is called {\em Baer} if $\Ext RMT = 0$ for all torsion modules $T$, and $M$ is {\em Mittag--Leffler} in case the canonical map $M\otimes_R \prod _{i\in I}Q_i\to \prod _{i\in I}(M\otimes_RQ_i)$ is injective where $\{Q_i\}_{i\in I}$ are arbitrary left $R$--modules.

We show that a module $M$ is Baer iff $M$ is $\p$--filtered where $\p$ is the preprojective component of the tame hereditary algebra $R$. We apply this to prove that the universal localization of a Baer module is projective in case we localize with respect to a complete tube. Using infinite dimensional tilting theory we then obtain a structure result showing that Baer modules are more complex then the (infinite dimensional) preprojective modules. In the final section, we give a complete classification of the Mittag--Leffler modules.
\end{abstract}

\maketitle

Since the fundamental work of Ringel \cite{R}, the study of infinite dimensional modules has become one of the challenging
tasks of the representation theory of finite dimensional hereditary algebras. In the present paper, we consider in detail
two classes of infinite dimensional modules over tame hereditary algebras: the Baer modules, and the Mittag-Leffler ones.

Besides proving structure results, we investigate further the surprising analogy with modules over Dedekind domains discovered in \cite{R}.
Indeed, the current progress relies heavily on applications of the recent powerful set--theoretic and homological methods developed originally for modules over domains, cf.\ \cite{ABH} and \cite{GT}.

\medskip

Baer modules can be defined in a rather general setting \cite{O3}:
Let $R$ be a ring and $\mathcal T$ be a torsion class in $\rmod
R$. A module $M \in \rmod R$ is a {\em Baer module} for $\mathcal
T$ provided that $\Ext RMT = 0$ for all $T \in \mathcal T$.

In the particular case when $R = \mathbb Z$, and more generally when $R$ is a integral domain and
$\mathcal T$ is the class of all torsion $R$-modules, we obtain thus the classical notions of a
Baer group \cite{B} and a Baer module \cite{K}. It took quite a long time to prove that these
notions actually coincide with the well--known notions of a free group and a projective module,
respectively. Countable Baer groups were shown to be free already in 1936 by Baer \cite{B},
but the arbitrary ones only in 1969 by Griffith \cite{G}. The projectivity of the classical Baer
modules was also shown in two steps spreading over decades, however, in a different order. First,
a reduction to countably presented modules was proved by set--theoretic methods by Eklof, Fuchs
and Shelah \cite{EFS} in 1990. Finally, the countably presented case has recently been settled in
\cite{ABH}.

In the present paper, we consider Baer modules over tame hereditary algebras, so our $\mathcal T$
is the class of all torsion modules in the sense of Ringel \cite{R} (see below for unexplained
terminology). Baer modules in this sense have been studied since 1980's, notably by Okoh and Lukas
in \cite{O2}, \cite{O3}, and \cite{L2}. Their focus was on countably presented modules, but their
results already indicated the complexity of the general case.

Here we take a different approach. First, in Section 1, we apply set--theoretic methods to prove
a reduction to finitely presented modules. We infer that the Baer modules are precisely the modules filtered by the finitely generated preprojective modules (Theorem \ref{characterization}).
This reveals the main difference from the integral domain case: the role of the finitely generated
projectives is taken by the much more complex class of the finitely generated preprojectives.

However, as shown in Section 2, the reduction is still sufficient to imply that Baer modules are locally projective, in the sense that their
universal localizations, as introduced by Schofield in \cite{S1}, are projective in case we localize with respect to a set of
simple regular modules containing a complete clique (Theorem \ref{universal}).

In Section 3 we employ infinite dimensional tilting  theory in order to investigate the structure of the  infinite dimensional torsion--free,
and Baer modules. The key role is played here by an infinite dimensional tilting module,  the so called Lukas tilting module $L$.
It is a module which has no finite dimensional direct summands, not even pure-injective direct summands,
but it is noetherian when viewed as a module over its endomorphism ring (Corollary \ref{matrix}).

Baer modules are characterized as the kernels of homomorphisms between modules in $\Add L$, and torsion-free modules as the pure--epimorphic images of Baer modules with kernels in $\Add L$ (Corollaries \ref{corol} and \ref{corola}).

We recover Lukas' theorem on the structure of countably presented Baer modules \cite{L2}, and
extend it to arbitrary Baer modules (Corollary \ref{general}); then we establish  a bijective correspondence between equivalence classes of
Baer modules and isomorphism classes of (infinite dimensional) preprojective modules (Corollary \ref{corresp}).

\medskip

Also the Mittag-Leffler modules can be defined in a very general
setting \cite{rothmaler}, \cite{relml}: Given a class of left
$R$--modules $\mathcal Q$, we call a  module $M$ {\em ${\mathcal
Q}$-Mittag-Leffler} if the canonical map $M\otimes_R \prod _{i\in
I}Q_i\to \prod _{i\in I}(M\otimes_RQ_i)$ is injective for any family
$\{Q_i\}_{i\in I}$ of left $R$--modules in ${\mathcal Q}$. For
${\mathcal Q}=R$-Mod, this is just the well--known notion of a {\em
Mittag--Leffler module} from \cite{RG}.

If $R$ is a tame hereditary algebra then the connection with Baer modules is that every Baer module $M$ is $\mathcal C$--Mittag--Leffler where
$\mathcal C$ denotes the class of all left $R$-modules without non--zero preinjective direct summands.

For artin algebras, it is known that Mittag--Leffler modules
coincide with the {\em separable} modules, that is, the modules $M$
whose finite subsets are always contained in  finitely presented direct
summands \cite{Z}; we observe that they also coincide with
the strict Mittag-Leffler modules, that is, with locally
pure--projective modules in the sense of \cite{Az3}. We prove that in
the setting of tame hereditary algebras, the torsion--free separable
modules coincide with the pure submodules of products of
indecomposable preprojective modules (Proposition \ref{strict}),
while the torsion reduced separable modules coincide with the
locally split epimorphic images of  direct sums of indecomposable
finitely generated regular modules (Proposition  \ref{reducedML}).

Then we prove the main result of Section 4, a structure theorem for Mittag-Leffler modules (Theorem \ref{culmin}).


\bigskip

All our rings are unital, and by the unadorned term ($R$-)module we
mean  a right  module over the ring $R$.

We will use the following notation. For a class of modules $\mathcal C$, we define

$$^o \mathcal C = \{ M \in \rmod R \mid \Hom RMC = 0 \mbox{ for all } C \in
\mathcal C \},$$

$$^{\perp_1} \mathcal C = \{ M \in \rmod R \mid \Ext RMC = 0 \mbox{ for all } C \in
\mathcal C \}$$

and

$${^\perp} \mathcal C = \{ M \in \rmod R \mid \Exti iRMC = 0 \mbox{ for all } C \in
\mathcal C \mbox{ and all } i > 0 \}.$$

Similarly, the classes $\mathcal C ^o$, $\mathcal C ^{\perp_1}$ and $\mathcal C ^\perp$ are
defined.

\medskip

We will need some notions related to purity. A monomorphism $\iota: N\to M$ is said to be
 {\em s-pure} (or {\em strongly pure} or {\em locally split})
if for any finite subset $S \subseteq N$ there is an $R$-homomorphism $\pi : M \to N$ such that $\pi \iota(x) = x$
for each $x \in S$. According to \cite{Z}, a module $M$ is said to be {\em locally pure-injective} if every
pure-monomorphism $M\to X$ with $X\in \text{Mod-}R$ is s-pure.

An {\em s-pure submodule} $N$ of a module $M$  is a submodule with
the property that the embedding $\iota: N\to M$ is s-pure. Replacing
the term ``finite subset" by ``countable subset", we obtain the
notion of a {\em c-pure} submodule of a module $M$. Clearly, we have
the implications: $N$ is a direct summand in $M$ $\Rightarrow$ $N$
is c-pure in $M$ $\Rightarrow$ $N$ is s-pure in $M$ $\Rightarrow$
$N$ is pure in $M$.

We will also use the dual notions.
 A homomorphism $\pi:X\to M$ is a
{\it locally split epimorphism}
if for each finite subset $F\subseteq M$ there is a map $\varphi=\varphi_F:M\to X$ such that $m=\pi\varphi(m)$ for all $m\in F$.
A module $M$ is said to be \emph{strict Mittag-Leffler} (or {\em locally pure-projective}) if every pure-epimorphism $X\to M$ with $X\in \text{Mod-}R$
is locally split,
see \cite{RG} and \cite{Az3}.

\medskip

Finally, recall that a module $T$ is said to be  a ($1$-) {\em tilting module} if it
 satisfies

  (T1) proj.dim$(T) \leq 1$;

  (T2) $\Ext RT{T^{(I)}} = 0$ for each set $I$; and

  (T3) there is an exact sequence
$0 \to R \to T_0 \to T_1 \to 0$ where $T_0$ and $T_1$ are direct summands of a (possibly infinite) direct sum of copies of $T$.

\bigskip

\section{Baer modules and $\p$--filtrations}

We start with a  general setting  and show that
set--theoretic methods allow  to reduce the structure of Baer
modules to the countably generated ones.

Let $\sigma$ be an ordinal. An increasing chain of submodules,
$\mathcal M = ( M_\alpha \mid \alpha \leq \sigma )$, of a module
$M$ is called a {\em filtration} of $M$ provided that $M_0 = 0$,
$M_\alpha = \bigcup_{\beta < \alpha} M_\beta$ for all limit
ordinals $\alpha \leq \sigma$ and $M_\sigma = M$.

Given a class of modules $\mathcal C$ and a module $M$, a filtration $\mathcal M$ is a {\em $\mathcal
C$--filtration} of $M$ provided that $M_{\alpha + 1}/M_\alpha$ is isomorphic to some element of
$\mathcal C$ for each $\alpha < \sigma$. In this case we say that $M$ is {\em $\mathcal C$--filtered}.

Given a cardinal $\kappa$, a filtration $\mathcal M$ is a {\em $\kappa$--filtration} of $M$
provided that $M_{\alpha}$ is $< \kappa$--generated for each $\alpha < \sigma$.

\begin{thm}\label{reduction}
Let $R$ be an $\aleph_0$--noetherian ring and $\mathcal T$ be a torsion class in $\rmod R$ such
that $^{\perp_1} \mathcal T = {}^\perp \mathcal T$. Assume that either $\mathcal T$ consists of
modules of finite injective dimension, or ${}^\perp \mathcal T$ consists of modules of finite
projective dimension. Then the following conditions are equivalent for any module $M$:

\begin{enumerate}

\item $M$ is a Baer module for $\mathcal T$.

\item $M$ has a filtration $\mathcal M = ( M_\alpha \mid \alpha \leq \kappa )$ such that,
for each $\alpha < \kappa$,
$M_{\alpha + 1}/M_\alpha$ is a countably generated Baer module for $\mathcal T$.

\end{enumerate}
\end{thm}

\begin{proof}
1.\ implies 2. The case when ${}^\perp \mathcal T$ consists of modules of finite (and hence
bounded) projective dimension is a particular instance of \cite[Theorem 4.3.10]{GT} (for $\nu =
\aleph_0$ and $n$ = the common bound of the projective dimensions of all modules in ${}^\perp
\mathcal T$).

If $\mathcal T$ consists of modules of finite injective dimension, then (as $\mathcal T$ is closed
under arbitrary direct sums) there is $n < \omega$ such that $\mathcal T \subseteq \mathcal I _n$
where $\mathcal I _n$ denotes the class of all modules of injective dimension $\leq n$. We will
proceed by a modification of the proof of \cite[Theorem 10.1.5]{GT} (for $\mu = \aleph_0$). Since
the latter proof has an extra set--theoretic assumption of the Weak Diamond, we will give more
details below indicating how to avoid this assumption by employing the fact that $\mathcal T$ is
closed under arbitrary direct sums.

For each $i \leq n$, let $\mathcal A _i = {}^\perp (\mathcal T \cup \mathcal I_i )$. Then
$\mathcal A _i = {}^\perp \mathcal T \cap {}^\perp \mathcal I_i = {}^{\perp_1} \mathcal T \cap
{}^{\perp_1} \mathcal I_i$. Denote by $\mathcal Q _i$ a representative set of all countably
generated modules in $\mathcal A _i$. By downward induction on $i$, we will prove that each module
in $\mathcal A _i$ is $\mathcal Q _i$--filtered (our claim 2.\ is then the case of $i = 0$, since
${}^\perp \mathcal I_0 = \rmod R$.)

For $i = n$ we have $\mathcal A _i = {}^\perp \mathcal I_n$. From Baer's
criterion of injectivity and the fact that all syzygies of cyclic modules can be taken countably
generated (since $R$ is $\aleph_0$--noetherian) we deduce that the cotorsion pair $({}^\perp \mathcal I_n, \mathcal I_n)$ is generated by a class of
$<\aleph_1$ - presented modules.
So the claim   follows   from \cite[Theorem 4.2.11]{GT}.

The induction step from $i+1$ to $i$ for $i < n$ is proved by transfinite induction on the minimal
number $\lambda$ of $R$--generators of $M \in \mathcal A ^{i}$ as in \cite[Theorem 10.1.5]{GT}.
The only place where the Weak Diamond is used there is when $\lambda$ is regular and uncountable,
for selecting a subfiltration with consecutive factors in ${}^\perp \mathcal T$ from a
$\lambda$--filtration of $M$ consisting of modules in ${}^\perp \mathcal T$. However, this is
possible by \cite[Theorem 4.3.2]{GT} since $\mathcal T$ is closed under arbitrary direct sums.

The implication 2.\ implies 1.\ is a particular case of the well--known Eklof Lemma (see e.g.\
\cite[Lemma 3.1.2]{GT}).
\end{proof}

\bigskip

 {\em From now on,  unless stated otherwise, we will assume that $R$ is an indecomposable tame hereditary artin algebra with
 standard duality $D\colon {\rm mod-}R\to R{\rm -mod}$ between right
 and left finitely generated $R$-modules, respectively}.

\medskip

We recall the notation of \cite{RR} and \cite{R}: $\tube$ and $\p$ will denote representative
sets of
all indecomposable finitely generated regular, and preprojective, modules, respectively. Further,
$$\mathcal F = \tube^o = {}^\perp \tube$$ is the torsion-free class of all
{\em torsion-free modules}.  Note that
 $$\mathcal F \cap \rfmod R = \add\p,\quad\text{hence}\quad  \mathcal F = \varinjlim \add \p.$$
 Since $\mathrm{add}\,\tube$ is closed
under extensions, the corresponding torsion class, called the
class of all {\em torsion modules},  is $\Gen \tube$, see
\cite[3.5]{RR}.

A module $M$ is a {\em Baer module} provided that $M$ is a Baer module for $\Gen \tube$. The class
of all Baer modules is denoted by $\mathcal B$. So  $$\label{classes}\mathcal B = {}^\perp \Gen \tube \subseteq
\mathcal F = {}^\perp \tube, \quad\text{and}\quad \mathcal B \cap \rfmod R = \mathcal F \cap \rfmod R = \add \p$$

\begin{rem}\label{remtilti} Baer modules over tame hereditary algebras are torsion-free, and the same holds for the classical
Baer modules over commutative integral domains. However, this
fails in general: Let $R$ be a ring and $T$ be a tilting module.
Take $\mathcal T = T^\perp$. This is the torsion class consisting
of all modules generated by $T$, and  the corresponding
torsion-free class is $T^o$. Thus $T$ is a Baer module for
$\mathcal T$ which is not torsion-free.
\end{rem}

The modules having no non-zero homomorphism to $\p$ are called {\em $\mathcal P ^\infty$-torsion},
and their class is denoted by $\mathcal L$. By the infinitary version of the Auslander-Reiten
formula \cite{C2} $$\mathcal L = {}^o\p = \p^\perp$$ so $\mathcal L$ is a tilting torsion class.
There is a countably infinitely generated tilting module generating $\mathcal L$, called the {\em
Lukas tilting module}, and denoted by $L$, cf.\ \cite{KT}.

For a module $M$, we will denote by
$l(M)$ the $\mathcal P ^\infty$-torsion submodule of $M$, that is, $l(M)$ is the trace of $L$ in $M$.
The torsion--free class corresponding to $\mathcal L$ will be denoted by $\mathcal P$. This is the
class of all (possibly infinitely generated) {\em preprojective modules}
(Warning: in \cite{L1} and
\cite{L2}, preprojective modules are called `$\mathcal P ^\infty$-torsion-free'). We have
$$\mathcal P = L^o= (^o \p)^o \subseteq \tube^o = \mathcal F, \quad\text{and}\quad \mathcal P \cap \rfmod R  = \add \p$$
Moreover,  if $M$ is
a module such that $M \notin \mathcal L$, then $M$ has a non-zero finitely generated preprojective
direct summand (see \cite[Corollary 2.2]{R}). In particular, the preprojective indecomposable modules are precisely the copies of
modules from $\p$. Similarly, the finitely $\p$--filtered modules are exactly the modules in $\add\p$.

\medskip

In our particular setting, thanks to a result from \cite{O3}, we can improve the reduction Theorem \ref{reduction} and obtain the
following characterization of Baer modules:

\begin{thm}\label{characterization} Let $R$ be a tame hereditary artin algebra and $M$ be a module.
Then $M$ is a Baer module if and only if $M$ is $\p$--filtered.
\end{thm}

\begin{proof} Assume $M \in \mathcal B$. By Theorem \ref{reduction}, $M$ has a filtration $\mathcal M =
( M_\alpha \mid \alpha \leq \kappa )$ such that $M_{\alpha + 1}/M_\alpha$ is a countably generated
Baer module for each $\alpha < \kappa$. So w.l.o.g., we can assume that $M$ is countably
generated. Then, by \cite[Theorem 1.2]{O3}, $M = \bigcup_{n < \omega} M_n$ where $M_n$ is finitely
generated, and $M_{n+1}/M_n \in \mathcal F$ for each $n < \omega$. Since $\mathcal B \cap \rfmod R
= \mathcal F \cap \rfmod R = \add\p$, we infer that $M$ has a filtration with all consecutive
factors finitely generated and preprojective, and hence $M$ is $\p$--filtered.

For the if--part, notice that since $\p \subseteq \mathcal B$, also
$\add\p \subseteq \mathcal B$, that is, all finitely generated
preprojective modules are Baer. So the claim follows   Eklof's Lemma
\cite[Lemma 3.1.2]{GT}.
\end{proof}

\begin{rem} Theorem \ref{characterization} is analogous to the corresponding result for integral domains.
In that case, however, one has to replace ``preprojective" by
``projective", so the existence of the filtration $\mathcal M$
already yields a direct sum decomposition, and hence projectivity,
of $M$, cf.\ \cite{ABH} and \cite{EFS}.
\end{rem}

Being in a hereditary ring, a submodule of a Baer module is again
Baer. However, the existence of a filtration of $M$ as in Theorem
\ref{characterization} implies existence of a large family of Baer
submodules of $M$ such that also all consecutive factors in the
family are Baer:

\begin{cor}\label{hillbaer} Let $R$ be a tame hereditary algebra. Let $M$ be a Baer module with a
$\p$--filtration $\mathcal M$ as in Theorem
\ref{characterization}. Then there exists a family $\mathcal Y$ of
(Baer) submodules of $M$ such that

\begin{enumerate}
\item[(H1)] $\mathcal M \subseteq \mathcal Y$ (in particular, $0, M \in \mathcal Y$).
\item[(H2)] $\mathcal Y$ is closed under arbitrary sums and intersections.
\item[(H3)] If $N,P \in \mathcal Y$ and  $N \subseteq P$, then $P/N$ is a Baer module.
\item[(H4)] Let $N \in \mathcal Y$ and $X$ be a finite subset of $M$. Then there is $P \in \mathcal Y$
such that $N \cup X \subseteq P$ and $P/N \in \add\p$.
\end{enumerate}

\end{cor}

\begin{proof} This follows by an application of the Hill lemma \cite[Theorem 4.2.6]{GT} to the filtration
$\mathcal M$ (for $\kappa = \aleph_0$).
\end{proof}

\bigskip

\section{Universal localizations of Baer modules}

Though Theorem \ref{characterization} does not yield direct sum
decompositions of Baer modules, it can be used to prove that
universal localizations of Baer modules, in the sense of Schofield
\cite{S1}, always decompose in case we localize with respect to a
set of isomorphism classes of simple regular modules $\mathcal U$
containing at least one complete clique (that is, in case the
universal localization $R_{\mathcal U}$ is a hereditary noetherian
prime ring, see \cite[\S4]{C1}). In fact, in this case the localized
Baer modules are projective.

Recall that a {\em clique} is an equivalence class with respect to
the equivalence relation induced on the set of isomorphism classes
of simple regular modules by the relation $S\sim S'$ if $\Ext
RS{S'}\not=0$. In other words, two simple regular modules $S$ and
$S'$ belong to the same clique if and only if they are in the same
tube. All cliques are finite, and all but finitely many consist of
exactly one simple regular module.

Let $\mathcal U$ be a set of isomorphism classes of simple regular
modules. For every $S\in \mathcal U$ fix a presentation \[0\to
P_S\stackrel{f_S}{\to}Q_S\to S\to 0\] consisting of finitely
generated projective modules. The ring $R_{\mathcal U}$ can be
described as the universal localization of $R$ at the set of maps
between projective modules  $\Sigma =\{f_S\vert S\in \mathcal{U}\}$.

In \cite[Theorem~4.2]{C1}, Crawley-Boevey proved that
$R_\mathcal{U}$ is a hereditary noetherian prime ring provided that
$\mathcal U$ contains, at least, one complete clique. Otherwise
$\mathcal {U}$ is finite, and $R_\mathcal{U}$ is again a tame
hereditary artin algebra with the same center as $R$.

For  discussing the  behaviour of  Baer
modules under localization, we will need the following preliminary observations.

\begin{lem} \label{tor} Let $R$ be a tame hereditary algebra.
Let $\mathcal U$ be a set of isomorphism classes of simple regular
modules. Then:
\begin{itemize}
\item[(i)] \textrm{ \cite[2.2, \S 4]{C1}} The canonical map $R\to R_ {\mathcal U}$ is an inclusion,
so that there is an exact sequence
$$0 \to R \to R_{\mathcal U} \to N \to 0.$$
\item[(ii)] $\mathrm{Tor}_1^R(P,R_ {\mathcal U})=0= \mathrm{Tor}_1^R(P,N)$ for any finitely generated preprojective
$R$-module $P$.
\item[(iii)] $\mathrm{Tor}_1^R(M,R_ {\mathcal U})=0= \mathrm{Tor}_1^R(M,N)$ for any
Baer $R$-module $M$.
\end{itemize}
\end{lem}

\begin{proof}
To prove $(ii)$ we note that in the  short exact sequence of left
$R$--modules
$$0 \to R \to R_{\mathcal U} \to N \to 0$$
the module $N$ is a directed union of left $R$-modules that are finitely
$\mathcal V$--filtered where $\mathcal V = \{ \mbox{Tr}U \mid U \in
\mathcal U \}$ and  $\mbox{Tr}U$ denotes the transpose of $U$ (see
\cite[Theorem 3]{S2}).

Let $P$ be any finitely generated preprojective module. For each $U
\in \mathcal U$ we have by \cite[Lemma 4]{S2} that $\Tor
RP{\mbox{Tr}U} \cong \Hom RUP = 0$. This proves that $\Tor RPV = 0$
for each $V \in \mathcal V$, and hence $\Tor RPN = 0$. Since also
$\Tor RPR = 0$, we infer that $\Tor RP{R_{\mathcal U}} = 0$.

 $(iii)$
follows from $(ii)$, because  Baer modules are torsion-free, hence direct limits of finitely generated preprojective modules,  and $\mathrm{Tor}$ commutes with direct limits.
\end{proof}

\begin{thm}\label{universal} Let $R$ be a tame hereditary algebra.
Let $\mathcal U$ be a set of isomorphism classes of simple regular modules.
Denote by $R_{\mathcal U}$ the universal localization of $R$ with respect to $\mathcal
U$.
Let further $M$ be a Baer $R$-module, and  $M_{\mathcal U} = M \otimes_R R_{\mathcal U}$
its universal localization.

If $\mathcal U$ does not contain a complete clique, then $M_{\mathcal U}$ is a Baer $R_{\mathcal U}$--module.

If $\mathcal U$ contains at least
one complete clique, then $M_{\mathcal U}$ is a projective $R_{\mathcal U}$-module.
\end{thm}

\begin{proof}
By Theorem \ref{characterization}, $M$ has a a filtration $\mathcal
M = ( M_\alpha \mid \alpha \leq \kappa )$ such that $M_{\alpha +
1}/M_\alpha = P_\alpha$ is an indecomposable (finitely generated)
preprojective module for each $\alpha < \kappa$. Since, by
Lemma~\ref{tor}, $\Tor R{P_\alpha}{R_{\mathcal U}} = 0$ for each
$\alpha < \kappa$, the $R_{\mathcal U}$--module $M_{\mathcal U} = M
\otimes _R R_{\mathcal U}$ is a direct limit of the continuous
direct system of modules
$$\mathcal M _{\mathcal U} = (M_\alpha \otimes _R R_{\mathcal U}, f_{\beta \alpha} \mid \alpha \leq \beta \leq \kappa )$$
where $f_{\beta \alpha} = u_{\beta \alpha} \otimes _R R_{\mathcal U}$ and $u_{\beta
\alpha} : M_\alpha \to M_\beta$ is the inclusion. As $\mbox{Coker}(f_{\alpha +1, \alpha}) \cong
P_\alpha \otimes_R R_{\mathcal U}$, it remains to investigate the universal localization of the
indecomposable preprojective $R$--modules.

 We start with the case where $\mathcal U$ does not contain a complete clique. Then
 $R _{\mathcal U}$ is again a tame hereditary algebra, but with a smaller rank of the
Grothendieck group, see \cite[Theorem 4.2]{C1}. Proceeding by induction, we can assume that
 $\mathcal U=\{S\}$ for some simple regular module $S$ in a tube of rank at least two.
 Then, as shown in \cite[10.1]{GL}, \cite[6.9]{S},
   the preprojective component $\p_{\mathcal U}$ of $R _{\mathcal U}$  is
$$\p_{\mathcal U} = \p \,\cap\, {\rm Mod-}R_{\mathcal U}  = \{ P \otimes _R R_{\mathcal U}  \,\mid\,  P \in \p\}$$
So we conclude that $M_{\mathcal U}$ is filtered by finitely generated preprojective $R_{\mathcal U}$--modules, hence it is a
Baer $R_{\mathcal U}$--module.

 Now assume that $\mathcal U$ contains at least
one complete clique. We prove that in this case the universal
localization of any indecomposable   preprojective $R$--module
$P$ is   a nonzero projective $R_{\mathcal U}$--module. This will
show that $M_{\mathcal U}$ is a projective $R_{\mathcal U}$-module
 (see
\cite[Lemma 3.1.5]{GT}).

Let $P$ be an indecomposable   preprojective module. By assumption,
there is a tube $\Theta$ whose mouth is a subset of $\mathcal U$.
Then the embedding of $P$ into its injective envelope $E(P)$
factorizes through a finite direct sum $U=\bigoplus_{i \leq m} U_i$
of elements of $\Theta$. Making a pullback of a resolution of $U$
$$0 \to P_1\stackrel{f}{\to} P_0\to U \to 0$$
consisting of finitely generated projective modules with the
inclusion $\varepsilon \colon P\to U$ we obtain a commutative
diagram
$$
\begin{CD}
0                 @>>>                P_1                 @>>f'>                Q                      @>>>                        P                @>>>                0               \\
@.                                @|                                  @VV\varepsilon 'V                                               @VV\varepsilon V                                                        \\
0                 @>>>                P_1                 @>>f > P_0
@>>>              U             @>>> 0
\end{CD}
$$
where $Q$ is projective because $\varepsilon '$ is injective and the
ring is hereditary.

As, for each $i \leq m$, all simple regular composition factors of
$U_i$ are in $\mathcal U$, an inductive argument on the regular
composition length of $U$ shows that $U\otimes_R R_{\mathcal U}=0=
\mathrm{Tor}_R^1 (U,R_{\mathcal U})$. Hence $f\otimes _R R_{\mathcal
U}$ is bijective, and $\mathrm{Id}= (f\otimes _RR_{\mathcal
U})^{-1}(\epsilon' \otimes _RR_{\mathcal U}) (f'\otimes
_RR_{\mathcal U})$. We deduce that $P\otimes _RR_{\mathcal U}$ is a
projective $R_{\mathcal U}$-module. Moreover as, by Lemma~\ref{tor},
$P$ embeds in $P \otimes _R R_{\mathcal U}$ we deduce  that $P
\otimes _R R_{\mathcal U}$ is a non--zero projective $R_{\mathcal
U}$--module.
\end{proof}

Let us now  investigate the universal localization of the Lukas tilting $R$-module $L$.
If $\mathcal U$ contains at least one complete clique then the localization is a projective $R_{\mathcal U}$-module
by Theorem \ref{universal}. The interesting case is the one when $\mathcal U$ does not contain a complete clique, so
$R_{\mathcal U}$ is again a tame hereditary algebra:

\begin{thm} \label{lukaslocalized} Let $R$ be a tame hereditary algebra with the Lukas tilting $R$--module $L$.
Let $\mathcal U$ be a set of isomorphism classes of simple regular modules which does not contain a complete clique.
Denote by $R_{\mathcal U}$ the universal localization of $R$ with respect to $\mathcal U$.
Then the universal localization $L_{\mathcal U} = L \otimes_R R_{\mathcal U}$  of $L$ is equivalent to
the Lukas tilting $R_{\mathcal U}$--module $L^\prime$.
\end{thm}

\begin{proof}
As in the proof of Theorem~\ref{universal} we can assume that
$\mathcal U=\{S\}$ for a simple regular module $S$ in a tube of rank
at least two. Let $\p_{\mathcal U}$ denote a representative set of
all indecomposable preprojective $R_{\mathcal U}$-modules.

First, we claim that $\Ext {R_{\mathcal U}}{P^\prime}{L_{\mathcal U}} = 0$ for each $P^\prime \in \p_{\mathcal U}$.
By the infinitary version of the Auslander-Reiten formula, it suffices to prove that
$\Hom {R_{\mathcal U}}{L_{\mathcal U}}{P^\prime} = 0$ for each $P^\prime \in \p_{\mathcal U}$.
Since the map $R\to R_{\mathcal U}$ is a ring epimorphism \cite[p. 56]{S1},  this is equivalent to $\Hom R{L_{\mathcal U}}{P^\prime} = 0$ for each $P^\prime \in \p_{\mathcal U}$.

By Lemma~\ref{tor},  the sequence of right $R$-modules $0 \to L \to
L_{\mathcal U} \to L \otimes_R N \to 0$ is exact.

We know from  \cite[10.1]{GL} that $\p_{\mathcal U} = \p\, \cap\,
{\rm Mod-}R_{\mathcal U}$. Thus $P^\prime$ is also preprojective as
right  $R$-module, so $\Hom RL{P^\prime} = 0$. Similarly, $\Hom
R{L\otimes_RN}{P^\prime} \cong \Hom RL{\Hom RN{P^\prime}} = 0$ since
$N \cong R_{\mathcal U}/R$ (as a right $R$-module) is a directed
union of finitely $\mathcal U$-filtered modules (that have no
non-zero homomorphisms to $P^\prime$) by \cite[Theorem 3]{S2}. Thus
$\Hom R{L_{\mathcal U}}{P^\prime} = 0$, and our claim is proved.

Now, we prove that $L_{\mathcal U}$ is a tilting module. Condition (T1) is clear since $R_{\mathcal U}$ is hereditary.

Condition (T2) states
that $\Ext {R_{\mathcal U}}{L_{\mathcal U}}{L_{\mathcal
U}^{(I)}} = 0$ for any set $I$.
Note that
$L_{\mathcal U}$ is a Baer
$R_\mathcal{U}$-module by Theorem~\ref{universal}, thus it is $\p_{\mathcal
U}$-filtered by Theorem~\ref{characterization}.
So, by \cite[3.1.2]{GT} it suffices to
show that $\Ext {R_{\mathcal U}}{P^\prime}{L_{\mathcal U}^{(I)}} =
0$ for each finitely generated preprojective $R_{\mathcal
U}$--module $P^\prime$.  Since $P^\prime$
is a finitely presented $R_{\mathcal U}$--module, the latter is equivalent
to $\Ext {R_{\mathcal U}}{{P^\prime}}{L_{\mathcal U}} = 0$, and (T2) follows from our claim above.

For condition (T3), consider an exact sequence $0 \to R \to L_0 \to
L_1 \to 0$ where $L_0$ and $L_1$ are direct summands in some
$L^{(J)}$. By Lemma~\ref{tor}, the sequence  $0 \to R_{\mathcal U}
\to (L_0)_{\mathcal U} \to (L_1)_{\mathcal U} \to 0$ is exact in
$\rmod {R_{\mathcal U}}$, and $(L_0)_{\mathcal U}$, $(L_1)_{\mathcal
U}$ are direct summands in $L_{\mathcal U}^{(J)}$.

This shows that $L_{\mathcal U}$ is a tilting module; it is also a
Baer $R_\mathcal U$-module by Theorem~\ref{universal}, thus
$L_{\mathcal U}\in{}^\perp({L^\prime}\,^\perp)$.  Moreover, by our
claim above, $\Ext {R_\mathcal U}{L^\prime}{L_{\mathcal U}} = 0$
because $L^\prime$ is $\p_{\mathcal U}$-filtered. So $L_{\mathcal U}
\in \Add (L^\prime)$, hence $L_{\mathcal U}$ is equivalent to
$L^\prime$.
\end{proof}

\medskip

Theorem \ref{lukaslocalized} will be used in a forthcoming paper for describing all infinite dimensional tilting modules over a tame hereditary artin algebra.

 \bigskip

\section{The structure of Baer modules}

Throughout this section, let $R$ be a tame hereditary artin algebra.
We will need further  properties of the (infinitely generated)
$R$-modules.

We denote by $G$ the generic module. By \cite[5.5]{R}, for each $M \in \mathcal F$ there are a unique cardinal $\kappa$, called the {\em
(torsion-free) rank} of $M$, and an exact sequence $0 \to M \to G^{(\kappa)} \to N \to 0$ where $N$
is torsion regular.

Following \cite{O1}, we call  $\mathcal E = G^\perp$ the class of all {\em cotorsion
modules}. In analogy with the Dedekind domain case, Okoh proved that all
torsion-free cotorsion modules are pure-injective, cf.\ \cite[p.265]{O1}. The first part of the
following Proposition will give this as a consequence of (infinite dimensional) cotilting theory.

Let us first recall some basic notions. A
   {\em cotorsion pair} is a pair of classes of modules $(\mathcal A,\mathcal B)$ such that
$\mathcal A = {}^\perp \mathcal B$ and $\mathcal B = \mathcal A
^\perp$. If $\Scal$ is a set of right $R$-modules, we obtain    a
cotorsion pair $(\mathcal A,\mathcal B)$ by setting $\mathcal B =
\Scal ^{\perp}$ and  $\mathcal A = {}^{\perp }(\Scal ^{\perp })$. It
is called the cotorsion pair {\em generated}\footnote{Our
terminology follows \cite{GT}, hence it differs from  previous use.}
by $\Scal$. Notice that $\mathcal A = {}^{\perp }(\Scal ^{\perp })$
consists of all direct summands of $\Scal\cup\{R\}$-filtered modules
\cite[Corollary~3.2.4]{GT}.

Dually, if $\Scal$ is a set of right $R$-modules, we obtain    a cotorsion pair $(\mathcal A,\mathcal B)$ by setting
$\mathcal A= {}^{\perp }\Scal $ and $\mathcal B = (^{\perp }\Scal) ^{\perp}$. It is called the cotorsion pair {\em cogenerated} by $\Scal$.
Cotorsion pairs  $(\mathcal A,\mathcal B)$ (co)generated by some (co)tilting module  $M$ are called {\em (co)tilting cotorsion pairs}.

We will denote by   $(\mathcal C,\mathcal D)$ the   cotorsion pair in $R$-Mod studied in \cite{RR}.
It is generated  by  a representative set $\tube^\prime$ of all
indecomposable regular left $R$-modules, that is,
 $\mathcal D=(\tube^\prime)^\perp$ is the class of all {\it divisible} left $R$-modules. Moreover,  $(\mathcal C,\mathcal D)$ is cogenerated  by  a representative set $\q^\prime$ of all
indecomposable preinjective left $R$-modules, that is,
 $\mathcal C={}^\perp(\q^\prime)$ is the class of all  left modules without non-zero preinjective summands. Finally,  $(\mathcal C,\mathcal D)$
 is a tilting and cotilting cotorsion pair (co)generated by the (co)tilting module $W$ which we call the {\em Ringel tilting left $R$-module}. Note that $W$ is the direct sum
of the generic left $R$-module $G^\prime$ with a representative set of the Pr\"ufer left $R$-modules.
Thus $D(W)$ is the direct product
of  $G$ with a representative set of the adic right $R$-modules.

\begin{prop} \label{pairs}

Let $R$ be a tame hereditary algebra.
\begin{enumerate}

\item $(\mathcal F, \mathcal E)$ is the cotorsion pair generated by $G$. This is also the cotilting cotorsion pair
cogenerated by $D(W)$. In particular, $\mathcal F \cap \mathcal E = \Prod D(W)$, so all
torsion-free cotorsion modules are pure-injective.

\item $(\mathcal B, \mathcal L)$ is the tilting cotorsion pair generated by the tilting module $L$. It is also
generated by the set $\p$. In particular, $\mathcal B \cap \mathcal L = \Add L$, and $L$ has no
non--zero finitely generated direct summands.

\end{enumerate}

\end{prop}

\begin{proof} 1. We   show that $^\perp (G ^\perp) = \mathcal F$. Since $G \in \mathcal F$, we have
$^\perp (G ^\perp) \subseteq \mathcal F$. On the other hand, if $M \in \mathcal F$, then  $M
\hookrightarrow G^{(\kappa)}$ where $\kappa$ is the torsion-free rank of $M$. Since $G \in
{}^\perp (G ^\perp)$, also $M \in {}^\perp (G ^\perp)$. Moreover,  $\tube = \{D(X)\,\mid\, X\in\tube^\prime\}$, thus   $\mathcal F =
\Ker \Tor R{-}{\tube^\prime}$. So, we deduce from \cite[2.3]{AHT} that $\mathcal F$ is the cotilting torsion-free class cogenerated by the
cotilting module $D(W)$.
 This proves the first two claims. That
$\mathcal F \cap \mathcal E = \Prod D(W)$ follows from well-known
properties of cotilting cotorsion pairs (see e.g.\ \cite[\S
8.1]{GT}). As dual modules (and also cotilting modules) are pure
injective, we deduce that torsion-free cotorsion modules are
pure-injective.

2. By Theorem \ref{characterization} and \cite[3.2.4]{GT}
we have $\mathcal B={}^\perp (\p ^\perp)={}^\perp (L ^\perp)$,
and the first two claims follow. That $\mathcal B \cap \mathcal L =
\Add (L)$ follows from well--known properties of tilting cotorsion
pairs (see e.g.\ \cite[\S 5.1]{GT}). Finally, if $F \in \Add L$ is
finitely generated, then $F \in \mathcal B \cap \rfmod R = \add (\p
) \subseteq \mathcal P$, so $F \in \mathcal P \cap \mathcal L
=\mathcal P \cap  {}^o\p = 0$.
\end{proof}

\begin{ex}\label{Gadic}
{\rm We observed on page \pageref{classes} that
$\mathcal B\subseteq \varinjlim \add \p=\mathcal F$, and the
finitely generated modules in both classes coincide (with the
finitely generated preprojective modules). However, the two classes
are different -- that is, $G \in \mathcal F \setminus \mathcal B$ --
because the class of all cotorsion modules is not closed under
arbitrary direct sums by \cite{O1} (or because there exist
preprojective modules which are not Baer, see below). In particular,
$\mathcal B$ is not closed under direct limits.

Note furthermore that also  the adic modules belong to $\mathcal F \setminus \mathcal B$.
In fact, let $S_{\lambda}, \, \lambda\in {\mathbb P}$, be a representative set of simple regular modules,
 and let $S_{\lambda}[-\infty], \, \lambda\in {\mathbb P}$, be the corresponding adic modules. If we assume that
  $S_{\lambda}[-\infty]$ belongs to  $^\perp(\p^\perp)$, then we obtain from
\cite[3.3(a)]{L2}
that every non-zero torsion-free module $Y\in \p^\perp $ satisfies Hom$(S_{\lambda}[-\infty],Y)\not=0$.
But this
contradicts the fact that Hom$(S_{\lambda}[-\infty],S_{\nu}[-\infty])=0$ whenever $S_{\lambda},S_{\nu}$ are not in the same clique.}
\end{ex}

\begin{rem} Notice the analogies of Proposition \ref{pairs} with the case when $R$ is an integral domain:

1. In general, there are three different cotorsion pairs of $R$-modules: $(\mathcal T \mathcal F,
\mathcal R \mathcal E)$ where $\mathcal T \mathcal F$ is the cotilting class of all torsion-free
modules, $(\mathcal F _0, \mathcal E \mathcal E)$ where $\mathcal F _0$ is the class of all flat
modules, and $(\mathcal S \mathcal F, \mathcal M \mathcal E)$, the cotorsion pair generated by the
quotient field $Q$ of $R$. Always  $\mathcal S \mathcal F \subseteq \mathcal F _0$ (the equality
holds iff all proper factors of $R$ are perfect rings, \cite{BS}) and $\mathcal F _0 \subseteq
\mathcal T \mathcal F$ (the equality holds iff $R$ is a Pr\" ufer domain, \cite{FS}). The
cotilting module $C$ generating $\mathcal T \mathcal F$ can be taken of the form $\delta^*$ where
$\delta$ is the Fuchs divisible (tilting) module (see e.g.\ \cite[Example 5.1.2]{GT}). The three
cotorsion pairs coincide iff $R$ is a Dedekind domain; in this case, $C$ can be taken of the form
$T^* = \Hom RT{Q/R}$ where $T = Q \oplus Q/R$ is the divisible tilting module analogous to the
Ringel tilting module. (In the tame hereditary case, the analogs of the first and third cotorsion
pairs coincide by Proposition \ref{pairs}.1, but the second is the trivial cotorsion pair
$(\mathcal P _0,\rmod R)$).

2. All Baer $R$-modules are projective by \cite{ABH}, so the analogue of the tilting cotorsion pair
$(\mathcal B, \mathcal L)$ from Lemma \ref{pairs}.(2) is the trivial tilting cotorsion pair
$(\mathcal P _0,\rmod R)$, and the analogue of the Lukas tilting module is $R$. This is reflected in the behaviour of
 $(\mathcal B, \mathcal L)$ under universal localization at tubes, see Theorem \ref{universal}.
\end{rem}

\medskip

We now employ Proposition \ref{pairs} to investigate the structure of the torsion-free $R$-modules.
Let us start with the Lukas tilting module.

{}From Example \ref{Gadic} and the classification of the
indecomposable pure-injective $R$-modules we   conclude that $L$ has
no pure-injective direct summand.  In particular, $L$ is not locally
pure-injective, see \cite[2.1(5)]{Z}. On the other hand, $L$
satisfies the following finiteness condition.

\begin{cor}\label{matrix}
$L$ is noetherian over its endomorphism ring.
\end{cor}
\begin{proof}
   Since $L$ is
torsion-free without preprojective summands, $D(L)$ is  divisible
without preinjective summands, thus $D(L)\in{\mathcal
C}\cap{\mathcal D}=\Add W$, cf. \cite[9.4]{relml}. This shows that
$D(L)$ is $\Sigma$-pure-injective since so is $W$ by
\cite[10.1]{RR}.
Now we apply \cite[9.9]{relml} which states that a tilting module is
noetherian over its endomorphism ring if and only if its dual module
is $\Sigma$-pure-injective.
\end{proof}

\begin{cor}\label{corol} Let $R$ be a tame hereditary algebra. A module $M$ is a Baer module if and
only if there is an exact sequence
$0 \to M \to L_1 \to L_2 \to 0$ where $L_1, L_2 \in \Add L$.
\end{cor}

\begin{proof} This follows from Proposition \ref{pairs}.(2) and from the well--known characterization
of modules forming tilting cotorsion pairs (see e.g.\ \cite[Proposition 5.1.9]{GT}).
\end{proof}

\begin{cor}\label{corola}  The following statements are equivalent for a module $M$.
\begin{enumerate}
\item
 $M$ is torsion--free.
 \item
 $M$ is a pure-epimorphic image of direct sum of indecomposable preprojective modules.
 \item  $M$  occurs as the end term in a pure--exact sequence
$$0 \to N \to B \to M \to 0$$ with $B \in \mathcal B$
and  $N \in \Add L$.
\end{enumerate}
\end{cor}

\begin{proof} By Proposition \ref{pairs}.(1), each $M \in \mathcal F$ is a pure--epimorphic image of a direct sum
 of modules in $\add \p$, hence (1) implies (2).\\ For (2) $\Rightarrow (3)$, recall that by \cite[3.2.1]{GT} there is
 a special $\mathcal B$-precover of $M$, that is, an exact sequence
 $0 \to N \to B \to M \to 0$ where $B \in \mathcal B$
and  $N \in \mathcal L$.  Since the pure-epimorphism from (2) factors through it, the sequence is pure.
 Moreover, $N\in {\mathcal L} \cap {\mathcal B} =\Add L$ because $\mathcal B$ is closed under submodules.\\
The implication (3) $\Rightarrow (1)$
follows from the fact that $\mathcal B \subseteq \mathcal F$ and $\mathcal F = \varinjlim \add
\p$ is closed under pure-epimorphic images.
\end{proof}

\medskip

Next, we look at the relation between $\mathcal B$ and $\mathcal P$. Since $\mathcal P \cap \rfmod
R = \add\p$, we have $\mathcal B \cap \rfmod R = \mathcal P \cap \rfmod R$.
Moreover, all preprojective modules of countable rank are Baer \cite[4.3(4)]{L2}:

\begin{lem} \label{countable} Let  $M$ be a preprojective module
of countable rank. Then $M$ is a countably generated Baer module.
\end{lem}

\begin{proof} First, let $F$ be a submodule of $M$ of finite rank. We prove that $F$ is finitely generated.
Indeed, if $F \neq 0$ then $F \notin \mathcal L$, so $F = P_1 \oplus F_1$ where $0 \neq P_1 \in
\add\p$. Similarly, if $F_1 \neq 0$ then $F_1 = P_2 \oplus F_2$ for $0 \neq P_2 \in \add\p$
etc. Since $F$ has finite rank, there is $n < \omega$ such that $F_n = 0$, so $F = \bigoplus_{i
\leq n} P_i \in \add\p$ is finitely generated.

So w.l.o.g., we can assume that $M$ has countably infinite rank. Consider an exact sequence $0 \to M \to G^{(\omega)}$.
Then $M _n = M \cap G^{(n)}$  ($n < \omega$) are submodules of $M$ of finite rank, hence finitely generated, with the property that
$M = \bigcup_{n < \omega} M_n$, and the factors $M_{n+1}/M_n \subseteq G $ belong to $ \mathcal F \cap \rfmod R = \add\p$.
So $M$ is  countably generated, and $M \in \mathcal B$ by Theorem \ref{characterization}.
\end{proof}

On the other hand, clearly, the countably generated tilting module $L$ is Baer, but not
preprojective.  In fact, $L$ is the only obstacle for the classes of countable rank Baer and
preprojective modules to coincide. This comes from the following version of a result by Lukas
(\cite[Theorem 4.3(c)]{L2}):

\begin{prop} \label{Lukas}  \, Let  $M$ be a module of countable rank.
Then $M$ is a Baer module if and only if $M = N \oplus P$ where $N \in \Add L$ and $P$ is
countably generated preprojective.
\end{prop}

\begin{proof} The `if' part is clear from Lemma \ref{countable}. For the `only if' part,
let $M$ be a Baer module and consider the exact sequence $\varepsilon: 0 \to l(M) \to M \to M/l(M) \to 0$.
Then $l(M) \in \mathcal B \cap \mathcal L = \Add L$. Moreover,
the module $M/l(M) \in \mathcal P$ is of countable rank by \cite[Lemma B]{O3}, hence countably generated Baer by Lemma \ref{countable}.
Since $l(M) \in \mathcal L$,   Lemma \ref{pairs}.(2) shows that $\varepsilon$ splits.
\end{proof}

A weaker statement holds for arbitrary Baer modules.
In general, the $\mathcal P ^\infty$-torsion submodule $l(M)$ of a Baer module  $M$ is not a direct summand,
but we are going to see that  is always a c-pure submodule.
Actually, we prove a slightly stronger result:

\begin{thm}\label{c-separable}  \, Let  $M$ be a Baer module.
Then
 each countable subset $C$ of  $l(M)$   is contained in a countably generated submodule $D \subseteq l(M)$ which is a direct summand in $M$.
In particular, $l(M)$ is a c-pure submodule in $M$.
\end{thm}

\begin{proof}
By Proposition \ref{Lukas}, we can w.l.o.g.\ assume that $M$ is not
countably generated. Let $\mathcal M$ be an $\add\p$-filtration of
$M$ coming from Theorem \ref{characterization}. Let $\mathcal Y$ be
a family of Baer submodules of $M$ corresponding to $\mathcal M$ by
Corollary \ref{hillbaer}. Let $\mathcal Y _c$ denote the set of all
countably generated modules from $\mathcal Y$. By Proposition
\ref{Lukas}, for each $Y \in \mathcal Y _c$, the exact sequence
$$\exs {l(Y)}{\nu_Y}{Y}{}{Y/l(Y)}$$
splits where $\nu _Y$ is the inclusion of $l(Y)$ into $Y$.

We claim that $l(Y)$ is even a direct summand in $M$. Indeed, let $\rho _Y : Y \to l(Y)$ be such
that $\rho_Y \nu _Y = \mbox{id}_{l(Y)}$. By conditions (H1) and (H3), the module $M/Y$ is Baer, so
we have the exact sequence
$$
\begin{CD}
{\Hom RM{l(Y)}} @>{\Hom R{\mu_Y}{l(Y)}}>> {\Hom RY{l(Y)}} \to {\Ext R{M/Y}{l(Y)}} = 0
\end{CD}
$$
where $\mu _Y$ is the inclusion of $Y$ into $M$. So $\rho_Y$ extends to some $\pi_Y \in \Hom
RM{l(Y)}$, and $\pi_Y \mu _Y \nu _Y = \mbox{id}_{l(Y)}$ proving that $l(Y)$ is a direct summand in
$M$.

Now, if $C$ is a countable subset of $l(M)$ then there exist countably many maps $f_i \in \Hom RLM$ ($i < \omega$)
such that $C \subseteq S = \sum_{i < \omega} \mbox{Im}(f_i)$. Since $L$ is countably generated, by conditions
(H2) and (H4), there exists $Y \in \mathcal Y _c$ such that $S \subseteq Y$, and hence $C \subseteq S \subseteq l(Y)$.
By the first part of the proof, $D = l(Y)$ is a direct summand in $M$.
\end{proof}

\begin{rem} Notice the following peculiar property of Baer modules: each Baer module $M$ possess a $\p$--filtration $\mathcal M = ( M_\alpha \mid \alpha \leq \kappa )$ as in Theorem \ref{characterization}, and in fact $M$ has many other $\p$--filtrations by Corollary \ref{hillbaer}. In particular, $M$ is the well--ordered directed union of the $M_\alpha$'s ($\alpha < \kappa$). The exact sequences $\exs {l(M_\alpha)}{}{M_\alpha}{}{M_\alpha/l(M_\alpha)}$ ($\alpha \leq \kappa$) also form a well-ordered direct system. But unlike its middle term, $\mathcal M$, this system is not continuous in general, that is, a sequence indexed by a limit ordinal $\alpha$ need not be the direct limit of the sequences indexed by all $\beta < \alpha$. This is the price payed for
$M/l(M)$ being preprojective, since direct limits of preprojective modules need not be preprojective.

This phenomenon is best seen for the Lukas tilting module $L$: it is a directed union of a countable chain $( P_n \mid n < \omega )$ where all $P_n$'s are finitely generated preprojective, so $l(P_n) = 0$ for each $n < \omega$, but $l(L) = L$.
\end{rem}

Theorem \ref{c-separable} yields the following generalization of Proposition \ref{Lukas}:

\begin{cor} \label{general}  \,
Let $M$ be a Baer module. Then there exists a unique pure-exact sequence
$$\varepsilon: 0 \to N \hookrightarrow M \to P \to 0$$
such that $N= l(M) \in \Add L$ and $P$ is preprojective.
If $l(M)$ is countably generated, then $\varepsilon$ is split exact.
\end{cor}

\begin{proof}  By Theorem \ref{c-separable}, the sequence $\exs {l(M)}{}M{}{M/l(M)}$ is pure-exact, and
it is even split exact when $l(M)$ is countably generated.
Clearly $l(M) \in \mathcal B \cap \mathcal L = \Add L$. Uniqueness follows from the fact that
$(\mathcal L, \mathcal P)$ is a torsion pair in $\rmod R$.
\end{proof}

Corollary \ref{corola} shows that any preprojective module occurs as the
${\mathcal P}^\infty$-torsion-free part of some Baer module. Such presentation is non-trivial in the uncountable rank
case, since there are preprojective modules which are not Baer:

\begin{ex}\label{Q}
{\rm Let $Q$ be the direct product of a set of representatives of the
isoclasses from $\p$. Then $Q$ is preprojective (since $\mathcal P$ is closed
under arbitrary direct products), but not Baer (see \cite[Proposition 3.1]{O3}).
So, $\mathcal B$ is not closed under direct products.

Moreover, considering the exact sequence $0 \to N \to B \to Q \to 0$ corresponding to $Q$ by
Corollary \ref{corola}, we see that $B$ is an uncountably generated Baer module whose
$\mathcal P ^\infty$-torsion submodule is pure in $B$, but does not split. So Proposition
\ref{Lukas} cannot be extended to uncountably generated Baer modules. }
\end{ex}

Corollary \ref{general} suggests the following definition: two Baer modules $B_1$ and $B_2$ are
{\em equivalent} ($B_1 \sim B_2$ for short) provided that $B_1/l(B_1) \cong B_2/l(B_2)$.

\begin{lem}\label{equiv} Let $R$ be a tame hereditary artin algebra and $B_1 , B_2 \in \mathcal B$.
Then $B_1 \sim B_2$ if and only if there exist
modules $L_1, L_2 \in \Add L$ such that $B_1 \oplus L_1 \cong B_2 \oplus L_2$.
\end{lem}

\begin{proof} Assume $B_1 \sim B_2$, so $P = B_1/l(B_1) \cong B_2/l(B_2) \in \mathcal P$. Considering the
pull-back of the epimorphisms $B_1 \to P$ and $B_2 \to P$, we obtain the following commutative
diagram

$$
\begin{CD}
                  @.                                  @.                  0                       @.                          0                                                       \\
@.                                @.                                  @VVV                                                @VVV                                                        \\
                  @.                                  @.                  L_1                 @=                L_1                                                           \\
@.                                @.                                  @VVV                                                @VVV                                                        \\
0                 @>>>                L_2                 @>>>                X                       @>>>                        B_2                 @>>>                0               \\
@.                                @|                                  @VVpV                                               @VVV                                                        \\
0                 @>>>                L_2                 @>>>                B_1                 @>>>              P             @>>>                0               \\
@.                                @.                                  @VVV                                                @VVV                                                        \\
                  @.                                  @.                  0                       @.                          0
\end{CD}
$$
where $L_2 = l(B_1)$ and $L_1 = l(B_2)$ belong to $\Add L$. By Proposition \ref{pairs}.(2),
the middle row and the middle column split, so
$X \cong B_1 \oplus L_1 \cong B_2 \oplus L_2$.

The reverse implication follows by factoring out the $\mathcal P ^\infty$-torsion submodule
of $B_1 \oplus L_1 \cong B_2 \oplus L_2$.
\end{proof}

\begin{cor}\label{corresp}  There is a bijective correspondence
between equivalence classes of Baer modules and isomorphism classes of preprojective modules.
\end{cor}

\begin{proof} The correspondence is defined on representatives as follows: $B \mapsto B/l(B)$ for a Baer module $B$;
given a preprojective module $P$, we first consider a special $\mathcal B$-precover $B \to P$, and assign $P \mapsto B$.
If $B^\prime \to P$ is another special $\mathcal B$-precover of $P$ then $B \sim B^\prime$ (see Lemma \ref{equiv}),
and the claim is clear.
\end{proof}

Since general preprojective modules are not classified, Corollary \ref{corresp} leaves little hope for a complete
classification of all Baer modules.

\bigskip

\section{The structure of Mittag--Leffler modules}

Recall that a   module  $M$ over   an arbitrary  ring is said to be {\em
separable} if every finite subset of $M$ is contained in a finitely
presented direct summand of $M$.

The torsion-free separable  modules over a tame hereditary algebra
were studied by Lenzing and Okoh in \cite{Okohsep} and  \cite{OL}.
We will see that   they are precisely the torsion-free (strict)
Mittag-Leffler $R$-modules. This will then enable us to give  a
complete classification of all Mittag-Leffler $R$-modules.

We start by recalling some notions. If $M$ is a right $R$-module,
and ${\mathcal Q}$ is a class of left $R$-modules, we say that $M$
is a \emph{${\mathcal Q}$-Mittag-Leffler module} if the canonical
map
\[\rho \colon  M\bigotimes_R \prod _{i\in I}Q_i\to \prod _{i\in
I}(M\bigotimes_RQ_i)\] is injective for any family $\{Q_i\}_{i\in
I}$ of modules in ${\mathcal Q}$. For ${\mathcal Q}=R$-Mod we obtain
the notion of a \emph{Mittag-Leffler module} from \cite{RG}.

The notion of relative Mittag-Leffler modules appears naturally in
connection with cotorsion pairs. For example, if $R$ is a tame
hereditary algebra and    $\mathcal C$ denotes the class of all left
$R$--modules without non--zero preinjective summands then, from a
general result on cotorsion pairs \cite[9.5]{relml}, we get the following result

\begin{prop} Let $R$ be a tame hereditary artin algebra.
Every Baer right $R$-module is $\mathcal C$-Mittag-Leffler.
Moreover, a countably generated torsion-free module is Baer if and
only if it is $\mathcal C$-Mittag-Leffler.
\end{prop}

\medskip

{}The next Lemma shows that over an Artin algebra the notions
of separable module, Mittag-Leffler module and strict Mittag-Leffler
module coincide.

\begin{lem}\label{separable} Let $R$ be a ring. For a given right $R$-module $M$, consider the following statements:
\begin{itemize}
\item[(i)] $M$ is separable.
\item[(ii)] $M$ is strict Mittag-Leffler.
\item[(iii)] $M$ is Mittag-Leffler.
\end{itemize}
Then  $(i)$ implies $(ii)$, and $(ii)$ implies $(iii)$. If all finitely presented right $R$-modules are pure injective,
then the three statements are equivalent.
\end{lem}

\begin{proof} In order to show $(i)\Rightarrow (ii)$ let $\pi \colon X\to
M$ be a pure epimorphism. If $F$ is a finite subset of $M$ then, by
hypothesis,  $F$ is contained is a finitely presented direct summand
$N$ of $M$. Denote by $\varepsilon \colon N\to M$ the inclusion and
by $f \colon M\to N$ any map such that $f \varepsilon =\mathrm{Id}$.
As $N$ is pure projective, there exists $g\colon N\to X$ such that
$\pi g=\varepsilon$. Set $\varphi= gf$. Then
$\pi \varphi (x)= \pi (gf)\varepsilon (x)
=\varepsilon (x)=x$ for any $x\in F$. This show that $\pi $ is a
locally split epimorphism. Therefore, $M$ is strict Mittag-Leffler.

The implication $(ii)\Rightarrow (iii)$ is due to Raynaud and Gruson
\cite{RG}. The remaining part of the statement is due to Zimmermann
\cite[2.11]{Z}.
\end{proof}

\bigskip

We now aim at a classification of all Mittag-Leffler modules over a tame hereditary algebra $R$. This is obtained in several steps. We start with a classification of the torsion-free Mittag-Leffler modules. A partial result in this direction was already achieved in \cite[3.1]{Okohsep}.

\begin{prop}\label{strict} Let $R$ be a tame hereditary artin algebra.
The following statements are equivalent for a  module $M$.
\begin{enumerate}
\item $M$ is torsion-free (strict) Mittag-Leffler.
\item $M$ is a locally split epimorphic image of a direct sum of indecomposable preprojective modules.
\item $M$ is an s-pure submodule of a direct product  of indecomposable preprojective modules.
\end{enumerate}
Moreover, the torsion-free (strict) Mittag-Leffler modules form a
preenveloping class which is closed under products and pure
submodules.
\end{prop}
\begin{proof}
Let $P$ be the direct sum of the modules in $\p$. It is well known that
$P$ is noetherian over its endomorphism ring. This follows from the
fact that its dual module is $\Sigma$-pure-injective \cite[4.6]{Le2}
by using \cite[Proposition 3 and Observation 8]{ZHZ}. By
\cite[2.3]{via}, a torsion-free module is (strict) Mittag-Leffler if
and only if it belongs to the class
 ${\mathcal G}(P)$ of all locally split epimorphic images of modules
in Add$P$.
Moreover,  it follows from \cite[3.4]{via} that ${\mathcal G}(P)$ is preenveloping and
 coincides with the class of all s-pure submodules of  direct products
of finitely generated preprojective modules.
The stated  closure properties follow from condition (3) and  \cite[2.3.3]{RG} or \cite[2.3]{via}.
\end{proof}

\begin{ex} {\rm (1) The modules in   Proposition~\ref{strict} are preprojective, but not
necessarily Baer, see Example \ref{Q}. Also,
there are preprojectives that are not separable, see \cite[Introduction]{Okohsep}. A
characterization of pure-projective preprojectives can be found in
\cite[2.4]{Okohsep}. \\
(2) The module $L$ is an example of a Baer module that is not Mittag-Leffler,
because $L$ is countably generated and not pure-projective, see \cite[2.2.2]{RG}.}
\end{ex}

\bigskip

 Recall   that every right $R$-module is a direct sum $X\oplus Y$ of a reduced module $X$ and a divisible module $Y$, see \cite{R}.
Next, we  investigate the torsion reduced Mittag-Leffler modules.

\begin{prop}\label{reducedML} Let $R$ be a tame hereditary artin algebra.
The following statements are equivalent for a  module $M$.
\begin{enumerate}
\item $M$ is a torsion reduced (strict) Mittag-Leffler module.
\item $M$ is a locally split epimorphic image of a direct sum of indecomposable
finitely generated regular
modules.
\end{enumerate}
\end{prop}
\begin{proof}
 Recall that every (strict) Mittag-Leffler module is  {separable} by Lemma~\ref{separable}. \\
(1) $\Rightarrow$ (2): For every $x\in M$ there is a finitely
generated direct summand $M_x$ of $M$ containing $x$. Then it is
easy to see that the map $\bigoplus_{0\not=x\in M} M_x\to M$ given
by the canonical inclusions $M_x\to M$ is a locally split
epimorphism. Moreover, the finitely generated modules $M_x$  are
again torsion reduced, hence direct sums of modules from
$\tube$.\\
 (2)$\Rightarrow$(1):  Of course, the class of all torsion modules is closed under direct sums and epimorphic images.
 Moreover,
 using that a  module is divisible if and only if it belongs to $^o\tube$, we easily verify that locally split
 epimorphic images of
   modules in $\Add \tube$ are reduced. Hence $M$ is torsion reduced, and it is   (strict) Mittag-Leffler by
   \cite[2.3]{via}.
\end{proof}


\bigskip

Let us give a further
 characterization of the torsion-free, and torsion reduced modules,
respectively, that are Mittag-Leffler.
To this aim, we need some preliminaries.
Recall that an inverse system of sets
$(H_\alpha,h_{\alpha\,\gamma})_{\alpha\,\gamma\in I  }$ is said to
satisfy the \emph{Mittag-Leffler condition} if for any $\alpha \in I
$ there exists $\beta\ge \alpha$ such that
$h_{\alpha\,\gamma}(H_\gamma)= h_{\alpha\,\beta}(H_{\beta})$  for
any $\gamma \ge \beta$.

It is shown in \cite{RG} that a right $R$-module $M$ is
Mittag-Leffler if and only if there exists a   direct system of
finitely presented modules $(F_\alpha,u_{\beta \,\alpha})_{\beta
\,\alpha \in I  }$ with $M=\varinjlim F_\alpha$ such that
 the inverse system $(\Hom R {F_\alpha} B, \Hom R {u_{\beta
\,\alpha}} B)_{\beta \,\alpha \in I  }$ satisfies the Mittag-Leffler
condition for any right $R$-module $B$.

In the next  Proposition we present a refinement
of this result from \cite{relml}. It will allow us to establish a situation
in which we can restrict to   special choices of $B$ for proving that a
module is Mittag-Leffler.

\begin{prop} \label{Prop1}  \emph{\cite[6.2]{relml}}  Let $R$ be a ring. Let $\mathcal{B}$ be a class of right
$R$-modules closed under direct sums, and let $\mathcal{Q}$ be a
class of left $R$-modules. Assume there exists a   direct system of
finitely presented modules $(F_\alpha,u_{\beta \,\alpha})_{\beta
\,\alpha \in I  }$ with $M=\varinjlim F_\alpha$ such that
 the inverse system $(\Hom R {F_\alpha} B, \Hom R {u_{\beta
\,\alpha}} B)_{\beta \,\alpha \in I  }$ satisfies the Mittag-Leffler
condition for any right $R$-module $B\in \mathcal{B}$.

If for every $Q\in\mathcal{Q}$ and every $\alpha\in I$ there exists
a map $f_\alpha\colon F_\alpha\to B_\alpha$ such that
$B_\alpha\in\mathcal{B}$ and $f_\alpha\otimes_RQ$ is a monomorphism,
then $M$ is a $\mathcal{Q}$-Mittag-Leffler module.
\end{prop}

\begin{prop}\label{limit} Let $R$ be a ring, and let $\mathcal{S}$ be a class of finitely presented
modules. Let $B$   be a module such that $\mathcal{S}\subseteq
\mathrm{Add}\, (B)$.  If $M$ is a module in $\varinjlim \mathcal{S}$,
then the following statements are equivalent:
\begin{itemize}
\item[(1)] $M$ is Mittag-Leffler.
\item[(2)] There exists a   direct system of
modules  $(F_\alpha,u_{\beta \,\alpha})_{\beta
\,\alpha \in I  }$ in $\mathcal{S}$ with $M=\varinjlim F_\alpha$ such that
the inverse system
$(\Hom R {F_\alpha} B, \Hom R {u_{\beta
\,\alpha}} B)_{\beta \,\alpha \in I  }$ satisfies the Mittag-Leffler
condition.
\item[(3)] For any   direct system of
finitely presented modules $(F_\alpha,u_{\beta \,\alpha})_{\beta
\,\alpha \in I  }$ with $M=\varinjlim F_\alpha$, the inverse system
  $(\Hom R {F_\alpha} B, \Hom R {u_{\beta
\,\alpha}} B)_{\beta \,\alpha \in I  }$ satisfies the Mittag-Leffler
condition.
\end{itemize}
\end{prop}

\begin{proof} By \cite[2.1.5]{RG} and \cite[4.4]{relml} we only have to
prove that $(3)$ implies $(1)$.

Fix a direct system $(F_\alpha,u_{\beta \,\alpha})_{\beta \,\alpha
\in I  }$ such that $F_\alpha \in \mathcal{S}$, for any $\alpha \in
I$, and satisfying that $M=\varinjlim F_\alpha$. By
\cite[3.9]{relml}, for any $B'\in \mathrm{Add}\, (B)$ the inverse
system $(\Hom R {F_\alpha} {B'}, \Hom R {u_{\beta \,\alpha}}
{B'})_{\beta \,\alpha \in I }$ satisfies the Mittag-Leffler
condition. Since $F_\alpha \in \mathrm{Add}\, (B)$, we can apply
Proposition~\ref{Prop1}, taking  $f_\alpha =\mathrm{Id}\colon
F_\alpha \to F_\alpha$ and $\mathcal{Q}=R$-Mod, to deduce that $M$
is a Mittag-Leffler module.
\end{proof}

Now  we  come back to tame hereditary algebras.

\begin{cor}Let $R$ be a tame hereditary artin algebra. Let $P$
be the direct sum of the modules in $\p$, and let $T$ be the direct
sum of the modules in $\tube$.

\begin{itemize}
\item[(1)] A torsion-free module $M$ is Mittag-Leffler if and only if there
exists a    direct system of finitely presented modules
$(F_\alpha,u_{\beta \,\alpha})_{\beta \,\alpha \in I  }$ with
$M=\varinjlim F_\alpha$ such that
 the inverse system $(\Hom R {F_\alpha} P, \Hom R {u_{\beta
\,\alpha}} P)_{\beta \,\alpha \in I  }$ satisfies the Mittag-Leffler
condition.

\item[(2)] A torsion  reduced module $M$ is Mittag-Leffler if and only if there
exists a   direct system of finitely presented modules
$(F_\alpha,u_{\beta \,\alpha})_{\beta \,\alpha \in I  }$  with
$M=\varinjlim F_\alpha$ such that
 the inverse system $(\Hom R {F_\alpha} T, \Hom R {u_{\beta
\,\alpha}} T)_{\beta \,\alpha \in I  }$ satisfies the Mittag-Leffler
condition.
\end{itemize}
\end{cor}
\begin{proof}
The statements are  immediate
consequences of Proposition~\ref{limit}, because the torsion-free modules
are the modules in $\varinjlim \mathrm{add}\, \p$, and the torsion reduced
modules belong to $\varinjlim \mathrm{add}\, \tube$, see \cite[3.4]{RR}.
\end{proof}

\bigskip

We are now ready to give the announced classification.

\begin{thm}\label{culmin} Let $R$ be a tame hereditary artin algebra.
The (strict) Mittag-Leffler modules over $R$ are precisely the
modules of
 the form $X\oplus Y$ where $Y$ is a direct sum of indecomposable preinjective modules,
and there exists a unique pure-exact sequence
$$\varepsilon: 0 \to X' \hookrightarrow X \to X'' \to 0$$
such that $X'$ is a locally split epimorphic image of a direct sum of indecomposable
finitely generated regular
modules, and $X''$ is an  (s-)pure submodule of a product of indecomposable
preprojective
modules.
\end{thm}
\begin{proof}
First of all, we check that every module of the stated form is
(strict) Mittag-Leffler. The module $Y$ is pure-projective, hence
strict Mittag-Leffler.
 The modules $X'$ and $X''$ are strict Mittag-Leffler by Propositions \ref{strict} and \ref{reducedML}. Since the
 class of strict Mittag-Leffler
 modules is closed under pure extensions \cite[8.12]{relml}, it follows that $X$ and $X\oplus Y$ are strict
 Mittag-Leffler.\\
  Conversely, if $M$ is a (strict) Mittag-Leffler module, we write
 $M=X\oplus Y$ as direct sum of a reduced module $X$ and a divisible module $Y$.
 Recall that divisible modules are direct sums of indecomposable preinjective modules and of indecomposable infinite dimensional pure-injective
 modules, see \cite{R}.
 But no indecomposable infinite dimensional module is separable. Hence $Y$ is a direct sum of indecomposable preinjective modules.
As for $X$, we know from \cite[4.1]{R}
that   there is a unique pure-exact sequence
$$0\to \tube X\to X\mapr{\nu} X/ \tube X\to 0$$ where
$\tube X=\sum_{f\in\Hom RYX, Y\in\tube}\Img f \in \text{Gen}\,\tube$ is the trace of $\,\tube$ in $X$, and $X/ \tube X\in \Fcal$.
Since the class of reduced modules is closed under submodules, and the  class of Mittag-Leffler
 modules is closed under pure submodules, we deduce that $X'=\tube X$ is torsion reduced Mittag-Leffler, hence it has the
 stated form by
 Proposition \ref{reducedML}.\\
It remains to prove that $X''=X/ \tube X$ is Mittag-Leffler. Then it will have the stated form by
Proposition \ref{strict}.
 So, we check that $X/ \tube X$ is separable. Let $0\not=\overline{z}\in X/ \tube X$ with $z\in X$. Then $z$ is contained
 in a finitely
 generated direct summand of $X$, so we can write   $X=A\oplus A'\oplus B$ with $A$ finitely generated  torsion-free,
 $A'$ finitely generated torsion,
 and
 $z=a+a'\in A\oplus A'$.
 Since $\Hom R{\tube X}A =0$, we have $\tube X\subseteq A'\oplus B$, which shows that
 $X/ \tube X=\nu(A)\oplus \nu(A'\oplus B)$. Thus $\nu(A)$ is a
 finitely generated direct summand of $X/ \tube X$ which contains $\overline{z}=\nu(z)=\nu(a)$.
 This completes the proof.
\end{proof}

We recall that over a tame hereditary artin algebra, countably
generated Mittag-Leffler modules are just direct sums of
indecomposable finitely generated modules. So if $M$ is a countably
generated Mittag-Leffler module then, in the Theorem above, $Y$,
$X'$ and $X''$ are just the direct sums of the preinjective, regular
and preprojective indecomposable direct summands from the direct sum
decomposition of $M$.

Mittag-Leffler modules are directed unions of their countably
generated Mittag-Leffler pure submodules \cite{RG}. The theorem
above explains how the  structure of the countably generated
submodules is reflected in the whole module.

\bigskip



\end{document}